\documentclass[10pt, reqno]{amsart}
\usepackage[colorlinks,citecolor=blue]{hyperref}
\usepackage{amsmath}
\usepackage{amscd}
\usepackage{amsthm}
\usepackage{amssymb}
\usepackage{latexsym}
\usepackage{euscript}
\usepackage{epsfig}
\usepackage{graphics}
\usepackage{array}
\usepackage{enumerate}
\usepackage{color}
\usepackage{wasysym}
\usepackage{graphicx}
\usepackage{amsfonts}
\usepackage{mathrsfs}
\usepackage{caption}
\newcommand{\Mod}[1]{\ (\text{mod}\ #1)}

\newcommand{\Hmm}[1]{\leavevmode{\marginpar{\tiny%
$\hbox to 0mm{\hspace*{-0.5mm}$\leftarrow$\hss}%
\vcenter{\vrule depth 0.1mm height 0.1mm width \the\marginparwidth}%
\hbox to 0mm{\hss$\rightarrow$\hspace*{-0.5mm}}$\\\relax\raggedright
#1}}}

\numberwithin{equation}{section}
\newtheorem{Thm}{Theorem}
\newtheorem{Rmk}{Remark}

\theoremstyle{definition}

\newcommand{\bel}[1]{\begin{equation}\label{#1}}

\newcommand{\be}{\begin{equation}}

\newcommand{\ba}{\begin{eqnarray}}
\newcommand{\ea}{\end{eqnarray}}

\newcommand{\qe}{\end{equation}}

\newtheorem{thesis}{Thesis}
\newcommand{\btl}[1]{\begin{thesis}\label{#1}}
\newcommand{\et}{\end{thesis}}

\theoremstyle{theorem}

\newtheorem{satz}{Proposition}
\theoremstyle{corollary}

\theoremstyle{lemma}
\newtheorem{lemma}{Lemma}
\theoremstyle{definition}
\newtheorem{defi}{Definition}
\theoremstyle{proof}

\theoremstyle{remark}

\date{}
\begin{document}

\title[Signature, lifts and eigenvalues]{Signatures, lifts, and eigenvalues of graphs}

\author{Shiping Liu}
\address{Department of Mathematical Sciences, Durham University, DH1 3LE Durham, United Kingdom}
\email{shiping.liu@durham.ac.uk}

\author{Norbert Peyerimhoff}
\address{Department of Mathematical Sciences, Durham University, DH1 3LE Durham, United Kingdom}
\email{norbert.peyerimhoff@durham.ac.uk}

\author{Alina Vdovina}
\address{School of Mathematics and Statistics, Newcastle University,
NE1 7RU Newcastle-upon-Tyne, United Kingdom}
\email{alina.vdovina@ncl.ac.uk}

\begin{abstract}
We study the spectra of cyclic signatures of finite graphs and the corresponding cyclic lifts. Starting from a bipartite Ramanujan graph, we prove the existence of an infinite tower of $3$-cyclic lifts, each of which is again Ramanujan.
\end{abstract}

\maketitle
\section{Introduction}
Constructing infinite families of (optimal) expander graphs is a very challenging topic both in mathematics and computer science, which has received extensive attentions, see e.g. \cite{HLW06}. Bilu and Linial \cite{BL06} succeeded in constructing expander graphs by taking $2$-lift operations iteratively. In particular, they relate $2$-lifts of a base graph $G=(V, E)$ with signatures $s: E\rightarrow \{+1,-1\}$ on the set of edges $E$, and reduce the construction problem to finding a signature $s$ whose signed adjacency matrix $A^s$ has a small spectral radius. Furthermore, Bilu and Linial conjectured that every $d$-regular graph $G$ has a signature $s: E\rightarrow \{+1,-1\}$ such that all the eigenvalues of $A^s$ have absolute value at most the Ramanujan bound, $2\sqrt{d-1}$. In a recent breakthrough, Marcus, Spielman and Srivastava \cite{MSS,MSSICM} proved Bilu and Linial's conjecture for bipartite graphs affirmatively, by which they obtained an infinite family of bipartite Ramanujan graphs for every degree larger than $2$ via taking $2$-lift operations iteratively, starting with a complete bipartite graph.

In this note, by considering more general groups of signatures, especially cyclic groups, we prove that for each $i\in \{1,2,\ldots, k-1\}$ where $k\geq 2$, every $d$-regular graph $G$ has a $k$-cyclic signature $s$ such that the maximal eigenvalue of the $i$-th power of its $k$-cyclic signed adjacency matrix $A^{s,i}$ in the sense of Hadamard product is at most $2\sqrt{d-1}$ (see Theorem \ref{thm:main} in Section \ref{section:ramanujan} for the general case).
This generalizes Marcus, Spielman and Srivastava's result for $\{+1,-1\}$-signed adjacency matrices. In particular, this enables us to show that every bipartite Ramanujan graph $G$ can be used as the starting point of an infinite tower of $3$-cyclic lifts, $\cdots\rightarrow G_k\rightarrow G_{k-1}\rightarrow G_{k-2}\rightarrow\cdots\rightarrow G_1=G$, where each $G_i$ is again Ramanujan (Theorem \ref{thm:tower3} in Section \ref{section:ramanujan}).

Besides constructing expander graphs, the ideas around (general) signatures and lifts of graphs have been developed from various motivations, e.g. social psychology, Heawood map-coloring problem, matroid theory, mathematical physics. We defer a brief historical review about these interesting developments to Section \ref{section:history}.

We emphasize that the set of $3$-cyclic lifts is a restrictive class of $3$-lifts, which we like to explain briefly. Let $G=(V,E)$ be a finite graph. For any two vertices $u,v\in V$, we denote the corresponding edge by $\{u,v\}\in E$ if it exists. One can assign an orientation to it, say, directing from $u$ to $v$, in which case, we write $e=(u,v)$. The same edge with the opposite orientation is then written as $\bar e:=(v,u)$. The set of oriented edges is denoted by $E^{or}$. A $3$-cyclic signature is a map $s: E^{or}\to \{1,\xi, \overline{\xi}\}$, where $\xi=e^{2\pi i/3}\in \mathbb{C}$ and $\overline{\xi}$ is the conjugate of $\xi$, such that
\begin{equation}\label{eq:sigintro}
 s(\bar e)=\overline{s(e)}, \,\,\,\,\text{ for all } e=(u,v)\in E^{or}.
\end{equation}
For every oriented edge $e=(u,v)\in E^{or}$, the three possible values of $s(e)$ correspond to different local cyclic lifts, as shown in the following figures.

\begin{figure}[h]
\begin{minipage}[t]{0.3\linewidth}
\centering
\includegraphics[width=0.85\textwidth]{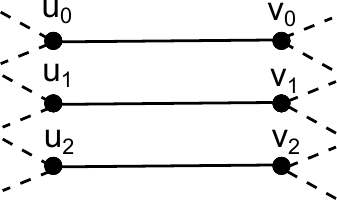}
\caption*{$s(e)=1$\label{L1}}
\end{minipage}
\hfill
\begin{minipage}[t]{0.3\linewidth}
\centering
\includegraphics[width=0.85\textwidth]{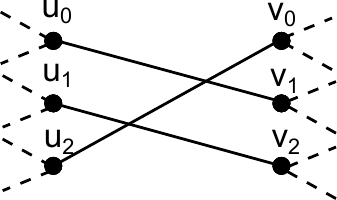}
\caption*{$s(e)=\xi$\label{L2}}
\end{minipage}
\hfill
\begin{minipage}[t]{0.3\linewidth}
\centering
\includegraphics[width=0.85\textwidth]{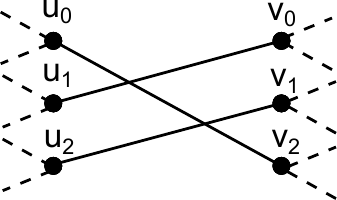}
\caption*{$s(e)=\overline{\xi}$\label{L2}}
\end{minipage}
\end{figure}

Let $A, A^s, \widehat{A}$ be the adjacency matrices of a graph $G$, its signature $s$, and the corresponding lift $\widehat{G}$, respectively. We will show that the eigenvalues $\sigma(\widehat{A})$ of $\widehat{A}$ satisfy (Lemma \ref{lemma:spectra of lifts})
\begin{equation}
  \sigma(\widehat{A})=\sigma(A)\sqcup\sigma(A^s)\sqcup\sigma(\overline{A^s}),
\end{equation}
where $\sqcup$ is the multiset union and $\overline{A^s}$ is the conjugate of $A^s$.

We prove the existence of our construction by applying the method of interlacing families in \cite{MSS} and mixed characteristic polynomials in \cite{MSS2} to our setting. The proof does not work for $k$-cyclic lifts, when $k\geq 4$. In this case, we have to find a signature $s$, such that all Hadamard powers of the associated signed adjacency matrix $A^s$ have simultaneously all their eigenvalues in the Ramanujan interval (see Lemma \ref{lemma:spectra of lifts}). For $k=3$, we only need to ensure that there is one signed adjacency matrix $A^s$ which satisfies this property. This is due to the fact that (\ref{eq:sigintro}) implies that $A^s$ is Hermitian and hence $\sigma(A^s)=\sigma(\overline{A^s})$. We like to mention that the $k$-cyclic lifts of a bipartite Ramanujan graph are a very special class and it is remarkable that such special lifts are sufficient to conclude that there are Ramanujan graphs in this class in the case $k=3$.

In this note, we consider the existence problem of cyclic signatures with particular spectral properties. For the special signature group $\{+1,-1\}$ more spectral theory of signed matrices can be found, e.g., in the survey paper \cite{ZaslavskyMatrices}. In a forthcoming paper \cite{LLPP}, we will extend results on Cheeger type constants and related spectral estimates, developed in \cite{AtayLiu14}, to the more general case of cyclic signatures.

\section{Basic notions and general framework}
Given a group $\Gamma$, which is usually finite, a general signature is defined as follows.
\begin{defi}\label{def:signature} A signature of $G=(V,E)$ is a map $s: E^{or} \to \Gamma$ satisfying
\begin{equation}\label{eq:keyforsig}
s(\bar e) = s(e)^{-1}, \,\,\,\, \text{ for all } e \in E^{or}.
\end{equation}
\end{defi}
For an oriented edge $e=(u,v)\in E^{or}$, we call $s(e)$ its signature, and write $s_{uv}$, alternatively. The signature of a cycle $C:=(u_1,u_2)(u_2,u_3)\cdots(u_{l-1},u_l)(u_l,u_1)$ is defined as the conjugacy class of the element
\begin{equation}
 s_{u_1u_2}s_{u_2u_3}\cdots s_{u_{l-1}u_{l}}s_{u_lu_1}\in \Gamma.
\end{equation}
\begin{defi}\label{def:balance}
 A signature $s$ of $G$ is called balanced if the signature of every cycle in $G$ is the identity element $\text{id}\in \Gamma$.
\end{defi}
Switching a signature $s$ by a function $\theta: V\rightarrow \Gamma$ means replacing $s$ by $s^{\theta}$, which is given by
\begin{equation}
 s^{\theta}(e):=\theta(u)s(e)\theta(v)^{-1},\,\,\,\,\text{ for all }e=(u,v)\in E^{or}.
\end{equation}
Two signatures $s$ and $s'$ of $G$ are called switching equivalent if there exists a function $\theta: V\rightarrow \Gamma$ such that $s'=s^{\theta}$. Switching equivalence between signatures is an equivalence relation. We denote the corresponding switching class of a signature $s$ by $[s]$.
Observe that being balanced is a switching invariant property.
\begin{satz}{\rm (\cite[Corollary 3.3]{Zaslavsky82})} A signature $s$ of $G$ is balanced if and only if it is switching equivalent to the signature $s_{\text{id}}$, where $s_{\text{id}}(e):=\text{id}$, for all $e\in E^{or}$.
\end{satz}

Signatures have interesting connections with lifts of graphs. In particular, we consider the permutation signatures, i.e. the maps $s:E^{or} \to {\mathcal S}_k$, where ${\mathcal S}_k$ denotes the group of permutations of $\{1,2,\ldots,k\}$. The $k$-lift $\widehat{G}=(\widehat{V}, \widehat{E})$ of $G=(V,E)$ corresponding to the permutation signature $s:E^{or} \to {\mathcal S}_k$ is defined as follows: The vertex set $\widehat{V}$ is given by the Cartesian product $V\times \{1,2,\ldots,k\}$. For any $u\in V$, we call $\{u_i:=(u,i)\}_{i=1}^k\subseteq \widehat{V}$ the fiber over $u$. Every edge $(u,v)\in E^{or}$ gives rise to the $k$ edges $(u_i, v_{s_{uv}(i)})$, $i=1,2,\ldots,k$, in $\widehat{E}^{or}$.

\begin{Thm}{\rm (\cite[Theorems 1 and 2]{GrossTucker77} and \cite[Theorem 9.1]{Zaslavsky82})}\label{zaslCorr} Let $G$ be a finite graph. There is a $1$-to-$1$ correspondence between the isomorphism classes of $k$-lifts of $G$ and the switching classes of signatures of $G$ with values in ${\mathcal S}_k$.
\end{Thm}
In particular, if two permutation signatures are switching equivalent, then the corresponding two $k$-lifts of $G$ are isomorphic. Observe that the $k$-lift of $G$ corresponding to a balanced permutation signature is composed of $k$ disjoint copies of $G$.

%
%

\section{Historical background}\label{section:history}
In 1953, Harary \cite{Harary53} introduced the concept of a signed graph, which is a graph $G=(V,E)$ with a signature $s:E\rightarrow \{+1,-1\}$, and the notion of balance (Definition \ref{def:balance}) in this setting. Harary was motivated by certain problems in social psychology, see also \cite{CartHarary56,Harary57,Harary59}. The switching equivalence of signatures was then described by the social psychologists Abelson and Rosenberg \cite{AR58}, and later discussed mathematically by Zaslavsky \cite{Zaslavsky82}.

Another source of the ideas around signatures and lifts is the Heawood map-coloring problem \cite{Heawood1890} asking for the chromatic number of a surface with positive genus, which is an extension of the famous four-color problem. The Heawood map-coloring problem is equivalent to finding the imbedding of every complete graph into a surface with the smallest possible genus \cite{Youngs67}. Gustin \cite{Gustin62} introduced the concept of a current graph in order to solve this imbedding problem, which was proved to be very important for the final solution due to Ringel and Youngs \cite{RingelYoungs68, Ringel74}. In the 1970s, Gross and Alpert \cite{GrossAlpert73, GrossAlpert74} developed Gustin's current graph theory into full topological generality and interpreted Gustin's method to construct an imbedding of a complete graph into a surface as a lift (or covering in topological terminology) of an imbedding of a smaller graph. Gross \cite{Gross74} further introduced the concept of a (reduced) voltage graph, which is a graph $G=(V,E)$ with a signature defined in Definition \ref{def:signature} (Gross called it a voltage assignment). A voltage graph can be considered as a dual graph of a current graph when both are imbedded into a certain surface. Gross associated to each signature $s: E^{or}\rightarrow \Gamma$ an $n$-lift of the graph $G$ with $n=|\Gamma|$, the order of $\Gamma$. Observe that voltage graphs are natural extensions of signed graphs of Harary.  One advantage of voltage graphs over current graphs is that their correspondence to lifts is independent of graph imbeddings, and the concept of a voltage graph makes the understanding of certain aspects of the solution of the Heawood map-coloring problem easier \cite{GrossTucker74}.

In \cite{GrossTucker77}, Gross and Tucker considered voltage graphs with $\Gamma=\mathcal{S}_k$ and established their correspondence to all $k$-lifts of $G$. In \cite{Gross74}, a cycle is called satisfying Kirchhoff's Voltage Law (KVL) if its signature is equal to the identity (compare with balance). Both KVL and its dual, the Kirchhoff's Current Law (KCL) in the current graph theory \cite{Gustin62}, play crucial roles in the corresponding lift and imbedding theory.

In 1982, motivated by a counting problem of the chambers of classical root systems, Zaslavsky \cite{Zaslavsky82} introduced the concepts of balance and switching equivalence of signatures into Gross and Tucker's theory on permutation signatures and lifts, and he formulated the explicit $1$-to-$1$ correspondence given in Theorem \ref{zaslCorr} above.

Connections between permutation signatures and lifts were also discussed by Amit and Linial \cite{AmitLinial02}, and they employed them to introduce a new model of random graphs. Friedman \cite{Friedman03} first used such a random model in the quest of finding larger Ramanujan graphs from smaller ones. This work stimulated an extensive study on the spectral theory of random $k$-lifts, see the recent work of Puder \cite{Puder14} and the references therein.

Agarwal, Kolla and Madan \cite[Section 1.1]{AKM13} pointed out that another motivation of considering permutation signatures and lifts is the famous Unique Game Conjecture of Khot. The permutations assigned to each oriented edge satisfying (\ref{eq:keyforsig}) appear naturally in the context of this conjecture.

We were led to consider general signatures and lifts by the notion of a discrete magnetic Laplacian studied in Sunada \cite{Sunada93} (see also Shubin \cite{Shubin94} and the references therein). This operator, originating from physics, is defined on a graph where every oriented edge has a signature in the unitary group $U(1)$ such that (\ref{eq:keyforsig}) holds. Sunada \cite{Sunada93} discussed switching equivalent signatures under a different terminology, cohomologous weight functions.

\section{Cyclic signature, lifts and adjacency matrices}
Let $S_k^1:=\{\xi^l \mid 0 \le l \le k-1\}$ be the cyclic group generated by the primitive $k$-th root of unity, $\xi = e^{2\pi i/k} \in {\mathbb C}$. We consider cyclic signatures, that is, maps $s: E^{or} \to S^1_k$. The corresponding signed adjacency matrix $A^{s}$ is a matrix with entries $(A^{s})_{uv}=s_{uv}$ if $\{u,v\}\in E$ and $0$ otherwise, where $u,v\in V$.
$A^s$ is Hermitian and has, therefore, only real eigenvalues with eigenvectors orthogonal w.r.t. the inner product $\langle a,b \rangle = \sum_{i=1}^k  a_i \bar b_i$.

The lift $\widehat{G}$ of $G$ corresponding to a $k$-cyclic signature is called
$k$-cyclic lift. In particular, every edge $(u,v)\in E^{or}$ with $s_{uv}=\xi^{l}$, for some $l\in \{0,1,\ldots, k-1\}$, gives rise to the following $k$ edges in $\widehat{G}$:
\begin{equation*}
(u_i, v_{i+l\Mod k}), i=0,1,\ldots, k-1.
\end{equation*}

The adjacency matrix $\widehat{A}$ of $\widehat{G}$ can be written as
\begin{equation}
\widehat{A}=\begin{pmatrix}
              A_0 & A_1 & A_2 & \cdots & A_{k-1} \\
              A_{k-1} & A_0 & A_1 & \cdots & A_{k-2} \\
              A_{k-2} & A_{k-1} & A_0 & \cdots & A_{k-3} \\
              \vdots & \vdots & \vdots & \ddots & \vdots \\
              A_1 & A_2 & A_3 & \cdots & A_0 \\
            \end{pmatrix},
\end{equation}
where $A_l$ is the adjacency matrix for the oriented edges $s^{-1}(\xi^l)$ and $A_l=A_{k-l}^T$. For $i\in \{0,1,2,\ldots, k-1\}$, let $A^{s,i}$ be the Hermitian matrix with entries
\begin{equation*}
(A^{s,i})_{uv}:=((A^s)_{uv})^i=(s_{uv})^i,
\end{equation*}
where $u,v\in V$. In particular, we have $A^{s,0}=A$, $A^{s,1}=A^s$. Observe that
\begin{equation}
A^{s,i}=\sum_{l=0}^{k-1}\xi^{il}A_l,
\end{equation}
and
\begin{equation}\label{equ:conjofhardamardpowers}
A^{s,i}=\overline{A^{s,k-i}}.
\end{equation}

\begin{lemma}\label{lemma:spectra of lifts}
The spectrum of $\widehat{A}$ is given by
\begin{equation}
\sigma(\widehat{A})=\bigsqcup_{l=0}^{k-1}\sigma(A^{s,i}),
\end{equation}
where the notion $\bigsqcup$ stands for the multiset union.
\end{lemma}
\begin{Rmk}
 This is an extension of Bilu and Linial \cite[Lemma 3.1]{BL06}. We will call $\bigsqcup_{l=1}^{k-1}\sigma(A^{s,i})$ the {\em new eigenvalues} of the lift. Lemma \ref{lemma:spectra of lifts} was formulated in a slightly different form in \cite[Theorem 5]{AKM13}. For the reader's convenience, we present a proof here.
\end{Rmk}

\begin{proof}
For any $i\in \{0,1,\ldots, k-1\}$, let $w_i$ be an eigenvector of $A^{s,i}$ with eigenvalue $\lambda$, i.e., $A^{s,i}w_i^T=\lambda w_i^T$.
Set $\widehat{w}_i:=(w_i, \xi^iw_i, \xi^{2i}w_i, \ldots, \xi^{(k-1)i}w_i)$. We check that
\begin{equation}
\widehat{A}\widehat{w}_i^T=\begin{pmatrix}
                             \sum_{l=0}^{k-1}A_l\xi^{il}w_i^T \\
                             \sum_{l=0}^{k-1}A_l\xi^{i(l+1)}w_i^T \\
                             \vdots \\
                             \sum_{l=0}^{k-1}A_l\xi^{i(l+k-1)}w_i^T \\
                           \end{pmatrix}
=\begin{pmatrix}
   A^{s,i}w_i^T \\
   A^{s,i}\xi^iw_i^T \\
   \vdots \\
   A^{s,i}\xi^{i(k-1)}w_i^T \\
 \end{pmatrix}=\lambda\widehat{w}_i^T.
\end{equation}
Therefore, $\lambda$ is also an eigenvalue of $\widehat{A}$ with eigenvector $\widehat{w}_i$.

Moreover, we have for any two eigenvectors $w_i, w_j$ of $A^{s,i}, A^{s,j}$, respectively, where $i\neq j$,
\begin{equation}
\langle \widehat{w}_i, \widehat{w}_j \rangle = \langle w_i, w_j\rangle(1+\xi^{i-j}+\xi^{2(i-j)}+\cdots+\xi^{(k-1)(i-j)})=0.
\end{equation}
Note that the number of mutually orthogonal eigenvectors of $\{A^{s,i}\}_{i=0}^{k-1}$ is $k|V|$. Therefore, we have
$\sigma(\widehat{A})=\bigsqcup_{l=0}^{k-1}\sigma(A^{s,i})$.
\end{proof}
As a consequence of Lemma \ref{lemma:spectra of lifts} and (\ref{equ:conjofhardamardpowers}), most of the new eigenvalues have even multiplicity.

\begin{lemma}\label{lemma:symmetric}
Let $G$ be a finite bipartite graph. Then, for any $i\in\{0,1,\ldots,k-1\}$ and any $s:E^{or}\rightarrow S_k^1$, the spectrum $\sigma(A^{s,i})$ is symmetric w.r.t. zero.
\end{lemma}
\begin{proof}
 First observe that for every $A^{s,i}$, there exists two square matrix $A_1$, $A_2$ such that
\begin{equation}
 A^{s,i}=\begin{pmatrix}
           0 & A_1 \\
           A_2 & 0 \\
         \end{pmatrix}.
\end{equation}
Furthermore, $A^{s,i}$ is Hermitian and has only real eigenvalues. Let $\lambda$ be an eigenvalue of $A^{s,i}$ with eigenvector $w:=(w_1,w_2)^T$. Then we have \begin{equation}
A_1w_2=\lambda w_1,\,\,\,\,A_2w_1=\lambda w_2,
\end{equation}
and we can check directly that $-\lambda$ is an eigenvalue of $A^{s,i}$ with the eigenvector $(w_1,-w_2)^T$.
\end{proof}

The following lemma is an extension of the corresponding result for signatures $s: E^{or}\rightarrow \{+1,-1\}$ in \cite[Proposition II.3]{ZaslavskyMatrices}.

\begin{lemma}\label{lemma:spectralswitching}
Let $s$ and $s'$ be switching equivalent. Then for each $i\in\{0,1,\ldots,k-1\}$, the matrices $A^{s,i}$ and $A^{s',i}$ are unitary equivalent, and hence have the same spectrum.
\end{lemma}
\begin{proof}
 Let $\theta: V\rightarrow S_k^1$ be the function such that $s'_{uv}=\theta(u)s_{uv}\overline{\theta(v)}$, for all $e=(u,v)\in E^{or}$.
Set $D^i(\theta)$ be the diagonal matrix with entries $(D^i(\theta))_{uu}=\theta(u)^i$, where $u\in V$. We can check that
\begin{equation}
 A^{s',i}=D^i(\theta)A^{s,i}\overline{D^i(\theta)}.
\end{equation}
\end{proof}

A set of $i$ edges is called an $i$-matching if no two of them share a common vertex. If $m_i$ denotes the number of $i$-matchings in $G$, then the matching polynomial of $G$ is defined as (see \cite{GG1981})
\begin{equation}
\mu_G(x):=\sum_{i=0}^{\lfloor \frac{N}{2} \rfloor}(-1)^im_ix^{n-2i}.
\end{equation}

Now we consider the signature $s$ as a random variable with the following properties. The signature of $(u,v)\in E^{or}$ and its inverse $(v,u)$ are chosen independently from the other oriented edges. The signature $s_{uv}$ is chosen uniformly from $S_k^1$ and this choice determines the value of  $s_{vu}=\overline{s_{uv}}$, as well.
We have the following proposition, extending a result of Godsil and Gutman \cite[Corollary 2.2]{GG1981} (see also \cite{MSS}).
\begin{satz}\label{lemma:matching poly}
For any $i\in \{1,2,\ldots, k-1\}$, the expectation of the characteristic polynomial of $A^{s,i}$ satisfies
\begin{equation}
\mathbb{E}_s (\det(xI-A^{s,i}))=\mu_G(x).
\end{equation}
\end{satz}
\begin{proof}
We denote by $\text{Sym}(S)$ the set of permutations of a set $S$, and by $[N]$ the set $\{1,2,\ldots, N\}$. Let $(-1)^{|\eta|}$ denote the signature of a permutation $\eta\in \text{Sym}([N])$. For $l\in \{0,1,2,\ldots, N\}$, we define a subset $P_l$ of $\text{Sym}([N])$ to be
\begin{equation*}
P_l:=\{\eta\in \text{Sym([N])}: \text{the number of indices } i\in [N] \text{ s. t. } \eta(i)\neq i \text{ is equal to } l\}.
\end{equation*}
Next, we calculate the characteristic polynomials of $A^{s,i}$:
\begin{align*}
&\det(xI-A^{s,i})=\sum_{\eta\in \text{Sym}([N])}(-1)^{|\eta|}\prod_{j=1}^N(xI-A^{s,i})_{j,\eta(j)}\\
=&\sum_{l=0}^N\sum_{\eta\in P_l}(-1)^{|\eta|}x^{N-l}\prod_{\substack{j=1\\\eta(j)\neq j}}^{N}(-A^{s,i})_{j,\eta(j)}\\
=&\sum_{l=0}^Nx^{N-l}\sum_{\substack{S\subseteq [N]\\|S|=l}}\sum_{\substack{\pi\in \text{Sym}(S)\\\pi(i)\neq i\, \forall i\in S}}(-1)^{|\pi|}\prod_{j\in S}(-A^{s,i})_{j,\pi(j)}.
\end{align*}
Observe that
\begin{equation}
\mathbb{E}_s((-A^{s,i})_{j,\pi(j)})=-\frac{1}{k}\sum_{l=0}^{k-1}\xi^{il}=0,
\end{equation}
and
\begin{equation}
\mathbb{E}_s((-A^{s,i})_{j,\pi(j)}(-A^{s,i})_{\pi(j),j})=\frac{1}{k}\sum_{l=0}^{k-1}\xi^{il}\overline{\xi^{il}}=1.
\end{equation}
Hence, we obtain
\begin{align*}
\mathbb{E}_s(\det(xI-A^{s,i}))=&\sum_{l=0}^{N}x^{N-l}\sum_{\substack{S\subseteq [N]\\|S|=l, l \text{ even}}}
\sum_{\substack{\pi \in \text{Sym}(S)\\ \pi(i)\neq i, \pi^2(i)=i\, \forall i\in S}}(-1)^{\frac{l}{2}}\\
=&\mu_G(x).
\end{align*}
\end{proof}
Heilmann and Lieb \cite{HeilmannLieb72} proved that for every graph $G$, $\mu_G(x)$ has only real roots and all these roots have absolute value at most $2\sqrt{d-1}$, where $d$ is the maximal vertex degree of $G$. A refinement in the irregular case was proved by Godsil \cite{Godsil81} leading to the following result presented in \cite[Lemma 3.5]{MSS}.
\begin{satz}\label{lemma:Mss}
Let $T$ be the universal cover of the graph $G$. Then the roots of $\mu_G(x)$ are bounded in absolute value by the spectral radius $\rho(T)$ of $T$.
\end{satz}

\section{Ramanujan properties}\label{section:ramanujan}

The following theorem is a generalization of \cite[Theorem 5.3]{MSS}.

\begin{Thm}\label{thm:main} Let $G =(V,E)$ be a finite connected graph.
Then for any $i\in \{1,2,\ldots, k-1\}$, there exists a cyclic signature $s^i_0: E^{or} \to S^1_k$ such that
\begin{equation}\label{eq:inmainthm}
\lambda_{max}(A^{s^i_0,i}) \le \rho(T),
\end{equation}
that is, all the eigenvalues of $A^{s^i_0,i}$ are at most the spectral radius $\rho(T)$ of the universal covering tree $T$ of $G$.
\end{Thm}

\begin{Rmk}
Note that by Lemma \ref{lemma:spectralswitching}, all the signatures in the switching class $[s_0^i]$ fulfill (\ref{eq:inmainthm}).
\end{Rmk}

For each $i\in \{1,2,\ldots, k-1\}$, we consider the following family of characteristic polynomials:
\begin{equation}\label{family:polynomials}
\{f^{s,i}:=\det(xI-A^{s,i})\mid s: E^{or}\rightarrow S_k^1\}.
\end{equation}
By Propositions \ref{lemma:matching poly} and \ref{lemma:Mss}, it is enough to prove the following property of (\ref{family:polynomials}):
\begin{align}
\begin{split}
&\text{there exists one polynomial of (\ref{family:polynomials}) whose largest root is no greater}\\
&\text{than the largest root of the sum of all polynomials in (\ref{family:polynomials}).}\label{property:key}
\end{split}
\end{align}
However, this property can not hold for an arbitrary family of polynomials. We apply the method of \emph{interlacing families}, developed by Marcus, Spielman and Srivastava \cite{MSS, MSS2, MSSICM} to prove property (\ref{property:key}).

First observe that for every signature $s: E^{or}\to S_k^1$, $f^{s,i}$ is a real-rooted degree $N$ polynomial with leading coefficient one. Let $\lambda_1(f^{s,i})\leq\lambda_2(f^{s,i})\leq \cdots\leq \lambda_N(f^{s,i})$ be the $N$ roots of $f^{s,i}$. If there exists a sequence of real numbers $\alpha_1\leq\alpha_2\leq\cdots\leq \alpha_{N-1}$ such that
\begin{equation}
\lambda_1(f^{s,i})\leq\alpha_1\leq\lambda_2(f^{s,i})\leq\alpha_2\leq \cdots\leq \alpha_{N-1}\leq \lambda_N(f^{s,i}) \,\,\,\,\,\forall s: E^{or}\rightarrow S_k^1,
\end{equation}
then we say that $\{f^{s,i}\}_s$ has a {\em common interlacing}. If the family of polynomials (\ref{family:polynomials}) could be proved to have a common interlacing, then property (\ref{property:key}) would hold by Lemma 4.2 in \cite{MSS}.

A systematic way to establish the existence of a common interlacing is given in the following lemma (see, e.g., \cite[Lemma 4.5]{MSS}).
\begin{lemma}\label{lemma:proveinterlacing}
Let $g^1, g^2, \ldots, g^l$ be polynomials of the same degree with positive leading coefficients. Then $g^1, g^2, \ldots, g^l$ have a common interlacing if and only if $\sum_{i=1}^lp_ig^i$ is real-rooted for all convex combinations, $p_i\geq 0$, $\sum_{i=1}^lp_i=1$.
\end{lemma}

In fact, in order to prove (\ref{property:key}), we do not prove that the polynomials $\{f^{s,i}\}_s$ have a common interlacing but that they form an \emph{interlacing family} introduced by Marcus, Spielman and Srivastava, for which we only need to consider special convex combinations of $\{f^{s,i}\}_s$ instead of all.

\begin{defi}[Interlacing families \cite{MSS}]\label{def:interlacing}
Let $S_1, \ldots, S_m$ be finite index sets and for every assignment $(s_1, \ldots, s_m)\in S_1\times S_2\times\cdots\times S_m$, let $g^{s_1,\ldots,s_m}(x)$ be a real-rooted degree $N$ polynomial with positive leading coefficient. For a partial assignment $(s_1,\ldots, s_q)\in S_1\times\cdots\times S_l$ with $1\leq q<m$, we define
\begin{equation*}
g^{s_1,\ldots, s_q}:=\sum_{s_{q+1}\in S_{q+1},\ldots, s_m\in S_m}g^{s_1,\ldots,s_q,s_{q+1},\ldots,s_m},
\end{equation*}
and
\begin{equation*}
g^{\emptyset}:=\sum_{s_{1}\in S_{1},\ldots, s_m\in S_m}g^{s_1,\ldots,s_m}.
\end{equation*}
The family of polynomials $\{g^{s_1,\ldots,s_m}\}_{s_1,\ldots, s_m}$ is called an \emph{interlacing family} if, for all $q\in \{0,1,\ldots,m-1\}$ and all given parameters $s_1\in S_1, \ldots, s_q\in S_q$, the family of polynomials
\begin{equation*}
\{g^{s_1,\ldots,s_q,t}\}_{t\in S_{q+1}}
\end{equation*}
has a common interlacing.
\end{defi}
Marcus, Spielman and Srivastava \cite[Theorem 4.4]{MSS} proved the following theorem.
\begin{Thm}\label{thm:Mss} Let $S_1,\ldots, S_m$ be finite index sets and let $\{g^{s_1,\ldots,s_m}\}_{s_1,\ldots, s_m}$ be an interlacing family of polynomials. Then there exists $(s_1, \ldots, s_m)\in S_1\times\cdots\times S_m$ such that the largest root of $g^{s_1,\ldots, s_m}$ is no greater than the largest root of $g^{\emptyset}$.
\end{Thm}

In order to prove property (\ref{property:key}) using Theorem \ref{thm:Mss}, we still need to prove the following proposition.
\begin{satz}\label{lemma:interlacing}
For each $i\in\{1,2,\ldots, k-1\}$, the family of polynomials $\{f^{s,i}\mid s: E^{or}\rightarrow S_k^1\}$ is an interlacing family.
\end{satz}
\begin{proof} For notational convenience, let $e_1,\ldots, e_m$ be all the oriented edges in $E^{or}$ and $s_1,\ldots, s_m$ their associated signatures, respectively. Then we can write the family of polynomials of this proposition as
\begin{equation*}
\{f^{s,i}\}_{s=(s_1,\ldots,s_m)\in (S_k^1)^m}.
\end{equation*}
Let $p_1^l,\ldots, p_m^l$, $l=0,1,\ldots, k-1$ be nonnegative real numbers satisfying
\begin{equation}\label{eq:convexity coeff}
\sum_{l=0}^{k-1}p_j^l=1, \,\,\,\,\text{for } j=1,2,\ldots, m.
\end{equation}
In order to prove this proposition, it is sufficient to prove that the following polynomial is real-rooted for all possible choices of $\{p_j^l\}$ satisfying (\ref{eq:convexity coeff}),
\begin{equation}\label{eq:partial convex comb}
\sum_{s=(s_1,\ldots,s_m)\in(S_k^1)^m}\left(\prod_{j=1}^mp_j^{l(s_j)}\right)f^{s,i}(x),
\end{equation}
where $l(s_j)\in \{0,1,\ldots, k-1\}$ satisfies $s_j=\xi^{l(s_j)}$.
In fact, if this real-rootedness is true, for each $q\in \{0,\ldots, m-1\}$ and fixed $s_1\in S_k^1,\ldots, s_q\in S_k^1$, we can apply Lemma \ref{lemma:proveinterlacing} to (\ref{eq:partial convex comb}) with
\begin{align*}
&p_{q+1}^l\geq 0,\,\,\text{for }l=0,1,\ldots,k-1,\,\,\sum_{l=0}^{k-1}p_{q+1}^l=1;\\
&p_{q+2}^l=\cdots=p_{m}^l=\frac{1}{k},\,\,\,\text{for } l=0,1,\ldots,k-1;\\
&p_j^l=\left\{
         \begin{array}{ll}
           1, & \hbox{if $s_j=\xi^l$;} \\
           0, & \hbox{otherwise,}
         \end{array}\,\,\,\,\text{for }j=1,2,\ldots,q,
       \right.
\end{align*}
to conclude that $\{f^{(s_1,\ldots,s_q,t),i}\}_{t\in S_k^1}$ has a common interlacing and hence Proposition \ref{lemma:interlacing} holds by Definition \ref{def:interlacing}.

Now we start to prove the real-rootedness of the polynomial (\ref{eq:partial convex comb}).
Observe that the matrix $A^{s,i}$ can be written as follows:
\begin{equation}
 A^{s,i}=\sum_{j=1}^mr_j^i\cdot (r_j^i)^{\ast}-D.
\end{equation}
In the above equation, we use the following notations: $D$ is the diagonal matrix with $D_{uu}=d_u$, for each $u\in V$; $r_j^i\in\mathbb{C}^N$ is a column vector associated to the signature $s_j$ of the oriented edge $e_j$. If $e_j=(u,v)$ for $u,v\in V$, we have
\begin{equation}
 r_j^i:=(0,\ldots,0, \alpha_j^i,0,\ldots,0,\overline{\alpha_j^i},0,\ldots,0)^T,
\end{equation}
where the non-zero entries are at the $u$-th and $v$-th positions, respectively, and $(\alpha_j^i)^2=(s_j)^i$. We use the notation that $(r_j^i)^{\ast}:=(\overline{r_j^i})^T$ for simplicity.

For each edge $e_j\in E^{or}$, we consider its signature $s_j$ as a random variable with values chosen randomly from $S_k^1$. All the $m$ random variables $s_1, \ldots, s_m$ are independent with possibly different distributions. In this viewpoint, the values $\{p_j^l\}_{l=0}^{k-1}$ in (\ref{eq:convexity coeff}) represent the distribution of $s_j$. Accordingly, the vectors $\{r^i_j\}_{j=1}^m$ are a set of independent finite-valued random column vectors in $\mathbb{C}^N$. Then, the polynomial (\ref{eq:partial convex comb}) is equal to the following expectation of characteristic polynomial:
\begin{equation}\label{eq:expectedChar}
 \mathbb{E}(f^{s,i})=\mathbb{E}(\det(xI-A^{s,i}))=\mathbb{E}\left(\det\left(xI+D-\sum_{j=1}^mr_j^i\cdot (r_j^i)^{\ast}\right)\right).
\end{equation}
If the graph $G$ is regular with vertex degree $d$, we have $D=dI$. Therefore $\mathbb{E}(f^{s,i})$ is the expectation of characteristic polynomials of a sum of independent rank one Hermitian matrices (with a shift of all roots by $-d$). In the terminology of \cite{MSS2}, the right hand side of (\ref{eq:expectedChar}) without the matrix $D$ is called the \emph{mixed characteristic polynomial} of the matrices
\begin{equation}
 A_j^i:=\mathbb{E}(r_j^i\cdot(r_j^i)^{\ast}), \,\,\,j=1,2,\ldots, m.
\end{equation}
Note that all the above matrices $A_j^i$ are positive semi-definite.
Then by \cite[Corollary 4.4]{MSS2}, the mixed characteristic polynomial of positive semi-definite matrices is real-rooted. This proves the real-rootedness of (\ref{eq:expectedChar}) in the regular case and hence the proposition.

In the case that $G$ is irregular, we can obtain the real-rootedness of (\ref{eq:expectedChar}) by modifying the arguments of \cite[Corollary 4.4]{MSS2}. For convenience, we outline the proof here. A proof similar to the one of \cite[Theorem 4.1]{MSS2} yields
\begin{align*}
&\mathbb{E}\left(\det\left(xI+x'D-\sum_{j=1}^mr_j^i\cdot (r_j^i)^{\ast}\right)\right)\\
=&\prod_{j=1}^{m}(1-\partial_{z_j})\det\left.\left(xI+x'D+\sum_{j=1}^mz_jA_j^i\right)\right|_{z_1=\cdots=z_m=0}.
\end{align*}
Therefore, we obtain
\begin{align*}
&\mathbb{E}(f^{s,i})=\mathbb{E}\left(\det\left(xI+D-\sum_{j=1}^mr_j^i\cdot (r_j^i)^{\ast}\right)\right)\\
=&\prod_{j=1}^{m}(1-\partial_{z_j})\det\left.\left(xI+x'D+\sum_{j=1}^mz_jA_j^i\right)\right|_{z_1=\cdots=z_m=0,x'=1}.
\end{align*}
Note that $\det(xI+x'D+\sum_{j=1}^mz_jA_j^i)$ is real stable by \cite[Proposition 3.6]{MSS2} and we conclude the real stability of $\mathbb{E}(f^{s,i})$ by \cite[Corollary 3.8 and Proposition 3.9]{MSS2}. Since real stability coincides with real rootedness in the case of univariate polynomials, we conclude that $\mathbb{E}(f^{s,i})$  is real-rooted. For more details, see \cite{MSS2}.
\end{proof}


\begin{Thm}\label{thm:tower3} Let $G$ be a finite connected bipartite graph. Then there exists a $3$-cyclic-lift $\widehat G$ of $G$ such that all its new eigenvalues lie in the Ramanujan interval $[-\rho(T),\rho(T)]$, where $\rho(T)$ is the spectral radius of the universal covering $T$ of $G$. In particular, when $G$ is $d$-regular, the interval is $[-2\sqrt{d-1}, 2\sqrt{d-1}]$.
\end{Thm}
\begin{proof}
 By Lemma \ref{lemma:spectra of lifts}, the new eigenvalues of the $3$-cyclic-lift $\widehat{G}$ are eigenvalues of either $A^{s,1}=A^s$ or $A^{s,2}$. From (\ref{equ:conjofhardamardpowers}) we know $A^{s,2}=\overline{A^s}$. Since $A^s$ is Hermitian, we obtain $\sigma(A^{s,2})=\sigma(A^s)$ for any choice of $s:E^{or}\rightarrow S_3^1$. Applying Theorem \ref{thm:main}, we can find an $s_0: E^{or}\rightarrow S_3^1$ such that
\begin{equation}
 \lambda_{max}(A^{s_0})\leq \rho(T).
\end{equation}
By Lemma \ref{lemma:symmetric}, $\sigma(A^s)$ is symmetric w.r.t. to zero since $G$ is bipartite. Therefore, we arrive at
\begin{equation}
 |\lambda_i(A^{s_0})|\leq \rho(T), \,\,\,\,\,|\lambda_i(A^{s_0,2})|\leq \rho(T), \,\,\text{for }i=1,2,\ldots, N.
\end{equation}
This proves the corollary.
\end{proof}

Starting from the complete bipartite graph $G_1:=K_{d,d}$, we can apply Theorem \ref{thm:tower3} repeatedly to obtain an infinite tower of $3$-cyclic lifts $\cdots\rightarrow G_k\rightarrow G_{k-1}\rightarrow G_{k-2}\rightarrow\cdots\rightarrow G_1$ with each $G_i$ being Ramanujan.

As we have commented in the Introduction, the above method of finding an infinite family of Ramanujan graphs does not work for $k$-lifts with $k \ge 4$. In this case, one needs to find a proper signature $s_0$ which works simultaneously for all $i\in \{1,2,\ldots,k-1\}$ in Theorem \ref{thm:main}.

\section*{Acknowledgements}
We thank Stefan Dantchev for bringing the reference \cite{AKM13} to our attention.
We acknowlege the support of the EPSRC Grant EP/K016687/1.

\end{document}